\documentclass{amsart}

\usepackage{epsfig}
\usepackage{rotating}
\usepackage{amsmath,amssymb,enumerate,graphicx}
\usepackage{subfigure}

\usepackage{tikz}
\usetikzlibrary{matrix,arrows}

\newtheorem{theorem}{Theorem}[section]
\newtheorem{definition}{Definition}[section]

\newtheorem{lemma}{Lemma}[section]

\newtheorem{corollary}{Corollary}[section]

\def\CC{{\mathbb C}}

\def\RR{{\mathbb R}}

\begin{document}
\allowdisplaybreaks
 
\title[Some Geometric Properties]{Some Geometric Properties of the Solutions of Complex Multi-Affine Polynomials of Degree Three}


%    author one information
\author{Chayne Planiden}
\address{The University of British Columbia, Okanagan  \\
3333 University Way, Unit 5 \\
Kelowna BC, V1V 1V7 Canada
}
\email{chayne.planiden@alumni.ubc.ca}
\thanks{The first author was partially supported by UBC Graduate Entrance Scholarship. }
%\thanks{}

%\author{Blagovest Sendov}
%\address{Bulgarian Academy of Sciences  \\
%Institute of Information and Communication Technologies \\
%Acad. G. Bonchev Str., bl. 25A \\
%1113 Sofia, Bulgaria}
%\email{acad@sendov.com}
%\thanks{The second author was partially supported by Bulgarian National Science Fund \#DTK 02/44. }
%%\thanks{}

%    author two information
\author{Hristo Sendov}
\address{Department of Statistical and Actuarial Sciences \\
Western University \\
1151 Richmond Str. \\
London, ON, N6A 5B7 Canada}
\curraddr{}
\email{hssendov@stats.uwo.ca}
%\thanks{The se author was partially supported by
%the Natural Sciences and Engineering Research Council (NSERC) of Canada.}

%    \subjclass is required.
\subjclass[2010]{Primary 30C10}

\date{}

\keywords{Polynomial, multiaffine symmetric polynomial, M\"obius transformation, variety, birational equivalence, polarization}

 \begin{abstract} 
In this paper we consider complex polynomials $p(z)$ of degree three with distinct zeros and their polarization $P(z_1,z_2,z_3)$ with three complex variables. 
We show, through elementary means, that the variety $P(z_1,z_2,z_3)=0$ is birationally equivalent to the variety $z_1z_2z_3 + 1 = 0$. Moreover, the rational map certifying the equivalence is a simple M\"obius transformation. 

The second goal of this note is to present a geometrical curiosity relating the zeros of $z \mapsto P(z,z,z_k)$ for $k=1,2,3$, where $(z_1,z_2,z_3)$ is arbitrary point on the variety $P(z_1,z_2,z_3)=0$.
\end{abstract}

\maketitle

\section{Introduction}
Denote by $\CC$ the complex plane and let $\CC^*:= \CC \cup \{\infty\}$. 
In this note, we consider polynomials 
$$
p(z) = (z-\alpha_1) (z-\alpha_2) (z-\alpha_3) = z^3 + a z^{2} + b z +c
$$ 
having distinct zeros $\alpha_1, \alpha_2,$ and $\alpha_3$. 
%(Refer to Lemma~\ref{lem-1.1} for the algebraic implications of the latter assumption.) 
The {\it polarization of $p(z)$ with three variables} is defined to be
\begin{align}
\label{gen-poly}
 P(z_1,z_2,z_3)= z_1z_2z_3+ \frac{a}{3}(z_1z_2+z_2z_3+z_3z_1)+\frac{b}{3}(z_1+z_2+z_3)+c.
\end{align}
Clearly, $p(z)= P(z,z,z)$ for all $z \in \CC$. It is easy to see that the polynomial $P(z_1,z_2,z_3)$ is irreducible in 
$\CC[z_1,z_2,z_3]$, hence the variety in $\CC^3$ defined by the equation $P(z_1,z_2,z_3)=0$
is irreducible,  see \cite[Chapter 6 \S 6, Exercise 11]{CLO:1992}. The variety $P(z_1,z_2,z_3)=0$ 
has no singular points, and its dimension is two, see \cite[Chapter 9 \S 6, Theorem 9]{CLO:1992}. 

One of the goals of this note is to show, through elementary means, that the variety $P(z_1,z_2,z_3)=0$ is birationally 
equivalent to the variety $z_1z_2z_3 + 1 = 0$. Moreover, the rational map certifying the equivalence is a simple M\"obius transformation. This is accomplished in Theorem~\ref{non-deg-corr}, but the main ingredient is Theorem~\ref{thm-reduction}. The second goal of this note is to present a geometrical curiosity relating the zeros of $z \mapsto P(z,z,z_k)$
for $k=1,2,3$, for any point $(z_1,z_2,z_3)$ on the variety, see Theorem~\ref{main-thm-2} or its equivalent Theorem~\ref{scc}.

The observations in this note have their roots in \cite{SS:2012}. That work initiated the investigation of the properties of the {\it loci} of complex polynomials. A {\it locus} of the polynomial (\ref{gen-poly}) is a closed subset of $\CC$, minimal with respect to inclusion, that contains at least one element from every solution $(z_1,z_2,z_3)$ of  (\ref{gen-poly}). A locus of a polynomial of degree $n$ is defined similarly through its polarization with $n$ complex variables. The notion of a locus allows for the formulation of extremal versions of several classical theorems about the zeros of complex polynomials, such as Grace's theorem, the Grace-Walsh-Szeg\H{o} coincidence theorem, the complex Rolle's theorem, and Laguerre's Theorem, see \cite{RS}. One of the main focuses of \cite{SS:2012} is the construction of several families of loci of the polynomial (\ref{gen-poly}).

%Let 
%\begin{align}
%\label{mob-trans}
%T(z)=(az+b)/(cz+d) \quad \mbox{with } ad-bc \not=0
%\end{align}
%be a nondegenerate M\"obius transformation. For every polynomial $p$ of degree $n$ we define
%\begin{align*}
%T[p](z):=(cz+d)^n p(T(z)).
%\end{align*}

%*************************************************************************************************************
\section{Preliminaries and notation}
%*************************************************************************************************************

Throughout this paper, the numbers $\alpha_1,\alpha_2$, and $\alpha_3$ are assumed to be distinct.
Denote the cubic roots of $-1$ by
$$
e_1 := e^{-i\pi/3}, e_2 := e^{i\pi/3}, \mbox{ and } e_3 := -1.
$$ 

\begin{lemma}
\label{lem-1.1}
Three complex numbers $\alpha_1$, $\alpha_2$, and $\alpha_3$ satisfy the relationship
\begin{align}
\label{rel-1}
(\alpha_1+\alpha_2+\alpha_3)^2 = 3(\alpha_1\alpha_2+\alpha_1\alpha_3+\alpha_2\alpha_3)
\end{align}
if and only if 
\begin{align}
\label{ifequitriag}
 \alpha_1e_1+\alpha_2e_2+\alpha_3e_3=0 \quad \mbox{ or } \quad \alpha_1e_2+\alpha_2e_1+\alpha_3e_3 = 0, \quad \mbox{but not both.}
\end{align}
Equivalently, the complex numbers $\alpha_1$, $\alpha_2$, and $\alpha_3$ are the vertices 
of an equilateral triangle. 
\end{lemma}

\begin{proof}
The proof is straightforward; just consider (\ref{rel-1}) as a quadratic equation with respect to $\alpha_3$ to see that the solutions
are $\alpha_1e_1+\alpha_2e_2$ and $\alpha_1e_2+\alpha_2e_1$. The last condition follows after observing that
$(\alpha_2-\alpha_1)e_2 = (\alpha_1e_1+\alpha_2e_2) - \alpha_1.$
\end{proof}

 A useful representation of $P(z_1,z_2,z_3)$ is given by
 \begin{align*}
 P(z_1,z_2,z_3) &= \Big(z_1z_2+ \frac{a}{3}(z_1+z_2) +\frac{b}{3} \Big) z_3 + \Big(  \frac{a}{3}z_1z_2 + \frac{b}{3} (z_1+z_2)+c \Big) \\
 &=: P_1(z_1,z_2)z_3 + P_2(z_1,z_2). 
 \end{align*}

The following two points will be of crucial importance. Let 
 \begin{align*}
u^* := -\frac{\alpha_1\alpha_2e_3+\alpha_1\alpha_3e_2+\alpha_2\alpha_3e_1}{\alpha_1e_1+\alpha_2e_2+\alpha_3e_3}
\mbox{ and }
v^*:= -\frac{\alpha_1\alpha_2e_3+\alpha_1\alpha_3e_1+\alpha_2\alpha_3e_2}{\alpha_1e_2+\alpha_2e_1+\alpha_3e_3}.
\end{align*}
Note that both $u^*$ and $v^*$ are finite if the zeros $\alpha_1, \alpha_2,$ and $\alpha_3$ are not vertices of an equilateral triangle. Otherwise, exactly one of of them is infinity. The points $u^*, v^*$ are distinct, since
\begin{align*}
u^* - v^* = -i4\sqrt{3}\frac{(\alpha_1-\alpha_2)(\alpha_1-\alpha_3)(\alpha_2-\alpha_3)}{(\alpha_1e_1+\alpha_2e_2+\alpha_3e_3)(\alpha_1e_2+\alpha_2e_1+\alpha_3e_3)} \not= 0.
\end{align*}

If the points $u^*, v^*$ are finite, we have the following representation, which is straightforward to verify directly
\begin{align}
\label{repres-Pu^*v^*}
P(z_1, z_2, u^*) &= \frac{1}{3} \frac{(\alpha_1e_2+\alpha_2e_1+\alpha_3e_3)^2}{(\alpha_1e_1+\alpha_2e_2+\alpha_3e_3)} (z_1-v^*)(z_2-v^*), \\
\label{repres-Pv^*u^*}
P(z_1, z_2, v^*) &= \frac{1}{3} \frac{(\alpha_1e_1+\alpha_2e_2+\alpha_3e_3)^2}{(\alpha_1e_2+\alpha_2e_1+\alpha_3e_3)} (z_1-u^*)(z_2-u^*).
\end{align}

%We have the following lemma.

A byproduct of the proof of Lemma~\ref{lem-20140227} is the expressions
\begin{align}
\label{repru^*v^*}
\{u^*, v^*\} = \Big\{\frac{(9c-ab)+i\sqrt{3} \sqrt{\Delta}}{2(a^2-3b)}, \frac{(9c-ab)-i\sqrt{3} \sqrt{\Delta}}{2(a^2-3b)}\Big\},
\end{align}
whenever $\alpha_1$, $\alpha_2$, and $\alpha_3$ are not vertices of an equilateral triangle,
where $\Delta$ is the {\it discriminant} of the polynomial $p(z)$:
\begin{align*}
\Delta := a^2b^2-4b^3-4a^3c-27c^2+18cab = (\alpha_1-\alpha_2)^2(\alpha_1-\alpha_3)^2(\alpha_2-\alpha_3)^2.
\end{align*}  

In the case when $\alpha_1$, $\alpha_2$, and $\alpha_3$ are vertices of an equilateral triangle, we have the following lemma. Since its proof is a direct computation, it is omitted.

\begin{lemma}
If $\alpha_1$, $\alpha_2$, and $\alpha_3$ are vertices of an equilateral triangle, then
\begin{align}
\label{maybe-last}
3(9c-ab) =  27c - a^3 \not= 0,
\end{align}
and
\begin{align}
\label{maybe-last-2}
-\frac{a}{3} = \frac{b^2-3ac}{9c-ab} = 
\left\{
\begin{array}{ll}
v^* & \mbox{if }  \alpha_1e_1+\alpha_2e_2+\alpha_3e_3=0,  \\
u^* & \mbox{if }  \alpha_1e_2+\alpha_2e_1+\alpha_3e_3=0.
\end{array}
\right.
\end{align}

\end{lemma}

%*************************************************************************************************************
\section{The bi-affine, symmetric, rational transformation}
%*************************************************************************************************************

The {\it bi-affine, symmetric, rational transformation} is the map  
 $$
 F: \CC^* \times \CC^* \rightarrow \CC^*,
 $$
obtained by solving
$P(z_1,z_2,z_3)=0$ for, say $z_3$:
\begin{align*}
F(z_1, z_2) : = - \frac{az_1z_2+b(z_1+z_2)+3c}{3z_1z_2+ a(z_1+z_2)+b}.
\end{align*}
Note that $F(z_1, z_2)$ generalizes the symmetric M\"obius transformations (that is, the M\"obius {\it involutions}).

By Lemma~\ref{2014-04-14-b}, $F(z_1, z_2)$ is well-defined whenever $(z_1, z_2) \not\in \{(u^*, v^*), (v^*, u^*)\}$, if $\alpha_1$, $\alpha_2$, $\alpha_3$ are not vertices of an equilateral triangle. If $\alpha_1$, $\alpha_2$, $\alpha_3$ are vertices of an equilateral triangle, then $F(z_1, z_2)$ is well-defined everywhere on 
$\CC^* \times \CC^*$. in particular, we have
$$
F(z_1, \infty) = - \frac{az_1+b}{3z_1+ a} \mbox{ and } T(\infty, \infty) = -\frac{a}{3}. 
$$
For a fixed $z_3$, the map $z \mapsto F(z,z_3)$, is a symmetric M\"obius transformation 
\begin{align*}
F(z, z_3) = - \frac{z(az_3+b) +(bz_3+3c)}{z(3z_3+ a)+ (az_3+b)}.
\end{align*}
%It is easy to verify that $z \mapsto F(z,z_3)$ is non-degenerate if and only if $z_3 \not= u^*, v^*$.
%When $z_2$ is fixed, the map $F(\cdot, z_2)$ is an ordinary, symmetric, M\"obius transformation.  

\begin{lemma}
\label{2014-04-14-a}
The M\"obius transformation $z \mapsto F(z,z_3)$ is non-degenerate if and only if $z_3 \not\in \{u^*, v^*\}$. Otherwise, we have
\begin{align}
\label{deg-F}
F(z_1,u^*) = v^* \quad \mbox{and} \quad F(z_1,v^*) = u^*
\end{align}
for all $z_1 \in \CC$.
\end{lemma}

\begin{proof}
The discriminant of  $F(\cdot, z_3)$ is
$$
(a^2-3b)z_3^2+(ab-9c)z_3 +b^2-3ac.
$$
If $\alpha_1$, $\alpha_2$, $\alpha_3$ are not vertices of an equilateral triangle, that is $a^2-3b \not = 0$, then
using (\ref{repru^*v^*}), it is easy to check that $u^*, v^*$ are the zeros of that quadratic function.
If $a^2-3b = 0$, then the only zero is either $u^*$ or $v^*$, as seen using (\ref{maybe-last-2}).

In the case  $a^2-3b \not = 0$, relationships (\ref{deg-F}) follow from Lemma~\ref{2014-04-14-b}:
$$
P_1(z_1, u^*)v^*+P_2(z_1,u^*) = P(z_1, u^*, v^*) = P_1(u^*, v^*)z_1+P_2(u^*, v^*) = 0.
$$
In the case  $a^2-3b = 0$, one needs to use (\ref{maybe-last-2}) in order to see that
(\ref{deg-F}) holds.
\end{proof}

Consider the M\"obius transformation
\begin{align}
\label{defn-W(z)}
W(z) := - \frac{z(\alpha_1\alpha_2 e_3 + \alpha_1\alpha_3 e_2 + \alpha_2\alpha_3 e_1)+(\alpha_1\alpha_2 e_3 + \alpha_1\alpha_3 e_1 + \alpha_2\alpha_3 e_2)}{z(\alpha_1 e_1 + \alpha_2 e_2 + \alpha_3 e_3) + (\alpha_1 e_2 + \alpha_2 e_1 + \alpha_3 e_3) }.
\end{align}
The M\"obius transformation (\ref{defn-W(z)}) is precisely the one defined by the equations $W(e_k) = \alpha_k$ for $k=1,2,3.$
It is non-degenerate, since the zeros $\alpha_1, \alpha_2$, and $\alpha_3$ are distinct. 
%It is shown in \cite{} that 
%$$
%W[p](u) = (\alpha_1-\alpha_2)^2(\alpha_1-\alpha_3)^2(\alpha_2-\alpha_3)^2(u^3+1).
%$$
One of the goals of this note is to clarify the relationship between the solutions of the polarization of $p(z)$ and those of the polarization of $z^3+1$. This is done in Theorem~\ref{non-deg-corr}.

The inverse transformation of $W(z)$ is
$$
W^{-1}(z) = -\frac{z(\alpha_1e_2+\alpha_2e_1+\alpha_3e_3)+(\alpha_1\alpha_2e_3+\alpha_1\alpha_3e_1+\alpha_2\alpha_3e_2)}{z(\alpha_1e_1+\alpha_2e_2+\alpha_3e_3)+(\alpha_1\alpha_2e_3+\alpha_1\alpha_3e_2+\alpha_2\alpha_3e_1)},
$$
so $u^*, v^*$ are precisely the points satisfying 
\begin{align}
\label{defnu^*v^*}
W^{-1}(u^*)=\infty \mbox{ and } W^{-1}(v^*)=0,
\end{align}
thus we have
%Clearly $u^* \not= v^*$. Explicitly, we have
%\begin{align*}
%u^* := -\frac{\alpha_1\alpha_2e_3+\alpha_1\alpha_3e_2+\alpha_2\alpha_3e_1}{\alpha_1e_1+\alpha_2e_2+\alpha_3e_3}, \quad
%v^*:= -\frac{\alpha_1\alpha_2e_3+\alpha_1\alpha_3e_1+\alpha_2\alpha_3e_2}{\alpha_1e_2+\alpha_2e_1+\alpha_3e_3}, \mbox{ and }
%\end{align*}
\begin{align}
\label{Winv-1}
W^{-1}(z) = -\frac{\alpha_1e_2+\alpha_2e_1+\alpha_3e_3}{\alpha_1e_1+\alpha_2e_2+\alpha_3e_3} \Big(\frac{z-v^*}{z-u^*}\Big).
\end{align}

The following result, in essence, appears in \cite{SS:2012}. Since we introduce a few essential changes, and since it is of great importance to us,  we include the proof in Section~\ref{sect-proofThm}.

\begin{theorem}
\label{thm-reduction}
For any $(u_1, u_2) \in \CC^* \times \CC^* \setminus \{(0, \infty),(\infty, 0)\}$ we have
\begin{align}
\label{quadr-ident-31}
F(W(u_1), W(u_2)) = W\Big(-\frac{1}{u_1u_2} \Big).
\end{align}
%In particular,  for any $u \in \CC^*$, we have
%$$
%F(W(u), W(u)) = Q(W(u))=  W\big(-1/u^2 \big).
%$$ 
\end{theorem}

In light of Lemma~\ref{2014-04-14-a}, we make the following definition.

\begin{definition} \rm
\label{defn-non-deg}
Solutions of $P(z_1, z_2, z_3)=0$ such that $z_k \not\in \{u^*, v^*\}$ for all $k=1,2,3$ are called {\it non-degenerate}.
\end{definition}

Representations (\ref{repres-Pu^*v^*}) and (\ref{repres-Pv^*u^*}) show that if one of the components of a solution of $P$ is equal to 
$u^*$, then another has to be equal to $v^*$, and the third one is free. That is, the degenerate solutions of $P$ are
$\{(u^*, v^*, z) : z \in \CC\}$.

\begin{theorem}
\label{non-deg-corr}
The non-degenerate solutions  of 
$P(z_1, z_2, z_3)=0$ are in one-to-one correspondence with the solutions of 
\begin{align}
\label{requced-eqn}
\begin{cases}
u_1u_2u_3&=-1, \\
u_3 &\not=  - \frac{\alpha_1e_2 + \alpha_2e_1 + \alpha_3e_3}{\alpha_1e_1 + \alpha_2e_2 + \alpha_3e_3}.
\end{cases}
\end{align}
 via $u_k = W^{-1}(z_k)$ for $k=1,2,3$. 
\end{theorem}

\begin{proof}
Let $(z_1, z_2, z_3)$ be a non-degenerate solution of $P$. By (\ref{defnu^*v^*}) we see that $u_k := W^{-1}(z_k) \not\in \{0,\infty\}$ for $k=1,2,3$. 
Let $z_3:=F(z_1,z_2)$ and apply $W^{-1}$ to both sides. By (\ref{quadr-ident-31}) we get
$u_1u_2u_3=-1$.

Suppose now $u_1u_2u_3=-1$, and let $z_k := W(u_k)$  for $k=1,2,3$. By (\ref{defnu^*v^*}) we get 
$$
F(z_1,z_2) = W\Big(-\frac{1}{u_1u_2} \Big). 
$$
Both sides of this equality are not $\infty$, or else by (\ref{defn-W(z)}) we obtain
$$
u_3  = -\frac{1}{u_1u_2} =  - \frac{\alpha_1e_2 + \alpha_2e_1 + \alpha_3e_3}{\alpha_1e_1 + \alpha_2e_2 + \alpha_3e_3},
$$
which is a contradiction. Thus, letting $z_3:=F(z_1,z_2)$ defines a solution of $P$.
%Finally, using representation (\ref{Winv-1}), we obtain
%\begin{align*}
%-\frac{\alpha_1e_2+\alpha_2e_1+\alpha_3e_3}{\alpha_1e_1+\alpha_2e_2+\alpha_3e_3} \Big(\frac{z_3-v^*}{z_3-u^*}\Big) = W^{-1}(z_3) = u_3 = - \frac{\alpha_1e_2 + \alpha_2e_1 + \alpha_3e_3}{\alpha_1e_1 + \alpha_2e_2 + \alpha_3e_3}.
%\end{align*}
%By Lemma~\ref{lem-1.1}, the fraction on the right-hand side is neither $0$ nor $\infty$, so cancelling it out, we reach the contradiction
%$v^* = u^*$.
\end{proof}

Solutions of $P$ of the kind $(z,z,w)$, containing at most two distinct components, are called {\it bi-solutions}. 

Fix a non-degenerate solution $(z_1,z_2,z_3)$ of $P$. For each $k=1,2,3$, let $f_{k,1}, f_{k,2}$ be the zeros
of $z \mapsto P(z,z,z_k)$. 
By the comment after Definition~\ref{defn-non-deg}, the bi-solution $(f_{k,i}, f_{k,i}, z_k)$ is 
also non-degenerate, for $k=1,2,3$, $i=1,2$. At the moment it is not clear exactly which zero of $z \mapsto P(z,z,z_k)$
is denoted by $f_{k,1}$ and which one by  $f_{k,2}$. To fix the order we need a clarification.

The argument of a complex number is always understood to be in $(-\pi, \pi]$, that is, we consider only its principle value. Since, for complex numbers, the equality $(z_1z_2)^\alpha = z_1^\alpha z_2^\alpha$ holds only as equality between sets, we make the following agreement. By $\sqrt{z_1z_2}$ we understand:
$$
\sqrt{z_1z_2} := \sqrt{z_1} \sqrt{z_2}.
$$

Let $u_k := W^{-1}(z_k)$ and let $u_{k,i}:=W^{-1}(f_{k,i})$  for $k=1,2,3$, $i=1,2$. By Theorem~\ref{non-deg-corr}, $(u_1, u_2, u_3)$ is a solution and  $(u_{k,i}, u_{k,i}, u_k)$ is a bi-solution of (\ref{requced-eqn}) for $k=1,2,3$, $i=1,2$. Thus, 
$$
u_{k,i}^2 = - \frac{1}{u_k} = u_\ell u_j, \quad \mbox{ where $\{\ell,j,k\} = \{1,2,3\}$}. 
$$
Now, for indices $\{\ell,j,k\} = \{1,2,3\}$, we define
$$
u_{k,1} := \sqrt{u_\ell u_j} \quad \mbox{and} \quad u_{k,2} := -\sqrt{u_\ell u_j}
$$
and consequently
\begin{align}
\label{order}
f_{k,i} := W(u_{k,i}) \quad \mbox{for $k=1,2,3$, $i=1,2$.}
\end{align}
With this notation, we have the following geometric surprise. 

\begin{theorem}
\label{main-thm-2}
Let  $(z_1,z_2,z_3)$ be any non-degenerate solution of $P$. 
Let $f_{k,1}$ and $f_{k,2}$ be the zeros
of $z \mapsto P(z,z,z_k)$ ordered as described in (\ref{order}), for $k=1,2,3$.
Define the following seven circles and a conic:
\begin{description}
\item[$C_1$] the circle determined by $f_{1,1}, z_2, z_3$;
\item[$C_2$] the circle determined by $z_1,f_{2,1}, z_3$;
\item[$C_3$] the circle determined by $z_1,z_2, f_{3,1}$;
\item[$C_4$] the circle determined by $f_{1,2}, f_{2,2}, f_{3,2}$;
\item[$C_5$] the circle determined by $f_{1,1}, f_{2,1}, f_{3,2}$;
\item[$C_6$] the circle determined by $f_{1,1}, f_{2,2}, f_{3,1}$;
\item[$C_7$] the circle determined by $f_{1,2}, f_{2,1}, f_{3,1}$;
\end{description}
These seven circles have a common intersection point.  Moreover, let 
\begin{description}
\item[$N$] be the conic determined by the points $\{W^{-1}(f_{k,i}) : k=1,2,3, i=1,2 \}$.
\end{description}
Then, the curve $W(N)$ also passes through the common intersection point of the seven circles.
\end{theorem}

\begin{proof}
Applying the M\"obius transformation $W^{-1}$ to the seven circles and the curve $W(N)$ transforms the statement of the theorem into the one of Theorem~\ref{scc}. Section~\ref{sect-mainthm} is dedicated to the proof of Theorem~\ref{scc}.
\end{proof}

The common intersection point, referred to in the theorem is $W(u_0)$, where $u_0$ is defined by (\ref{defn-AB}) and $u_k := W^{-1}(z_k)$, for $k=1,2,3$. Theorem~\ref{main-thm-2} is illustrated on 
Figure~\ref{IllThm} below. As seen on Figure~\ref{IllThm}, the curve $W(N)$ in Theorem~\ref{main-thm-2} may not be a conic section. For additional information about M\"{o}bius transformations of conic sections, see \cite{CF:2007}.

%**************************************************************************
\section{The  rational quadratic function $F(z,z)$}
\label{sect-biaffineMob}
%**************************************************************************

 Consider the rational quadratic function $Q(z) := F(z,z)$, or explicitly
\begin{align*}
Q(z) = - \frac{az^2+2bz+3c}{3z^2+ 2az+b}.
\end{align*}
Since $Q(z) = -P_1(z,z)/P_2(z,z)$, Corollary~\ref{lem-2014-02-25-a} shows that $Q$ is well-defined for all $z \in \CC^*$. (The numerator and the denominator cannot be simultaneously zero.)
%Define by continuity $Q(\infty) := -a/3$ and note that
%$$
%Q^{-1}(-a/3) = \Big\{\infty, -\frac{ab-9c}{2(a^2-3b)} \Big\}.
%$$
\begin{lemma}
\label{lem-20140227-b}
The map $Q : \CC^* \rightarrow \CC^*$ is onto. Every point in $\CC^*\setminus \{u^*, v^*\}$ has two distinct pre-images in $\CC^*$. In addition, we have
$$
Q^{-1}(u^*) = \{v^*\} \quad \mbox{and} \quad Q^{-1}(v^*) = \{u^*\}.
$$
\end{lemma}

\begin{proof}
The fact that $Q$ is onto is easy to see. 
Fix any $w \in \CC$. The equation $Q(z) =w$ is equivalent to 
$$
(a+3w)z^2+2(b+wa)z+(3c+wb) = 0
$$
with discriminant
$$
(a^2-3b)w^2+(ab-9c)w+(b^2-3ca).
$$
If $\alpha_1, \alpha_2$, and $\alpha_3$ are not vertices of an equilateral triangle, then the discriminant is
zero when $w$ is equal to $u^*$ or $v^*$, see (\ref{repru^*v^*}). Otherwise, by (\ref{maybe-last-2}) it is zero, when $w=-a/3 \in \{u^*, v^*\}$. The fact that $Q(v^*)=u^*$ and $Q(u^*)=v^*$ follows from Lemma~\ref{lem-20140227} and (\ref{maybe-last-2}).
\end{proof}

Next, we show that there is the unique symmetric M\"obius transformation $G$ that makes the diagram commute
\begin{center}
\begin{tikzpicture}[description/.style={fill=white,inner sep=2pt}]
\matrix (m) [matrix of math nodes, row sep=3em, column sep=2.5em, text height=1.5ex, text depth=0.25ex]
{ \CC^* &            & \CC^* \\
             & \CC^* & \\ };
\path[<->] (m-1-1) edge node[auto] {$ G $} (m-1-3);
\path[->] (m-1-1) edge node[auto, swap] {$ Q $} (m-2-2);
\path[->] (m-1-3) edge node[auto] {$ Q $} (m-2-2);
\end{tikzpicture}
\end{center}
It should be clear from Lemma~\ref{lem-20140227-b} that if such a transformation exists, its 
fixed points are $u^*$ and $v^*$. So, using (\ref{repru^*v^*}), define
\begin{align}
\label{fn-G}
G(z):=-\frac{(ab-9c)z+2(b^2-3ca)}{2(a^2-3b)z+(ab-9c)}.
\end{align}

\begin{lemma} \rm
\label{third-charact}
The symmetric M\"obius transformation $G(z)$ satisfies $Q \circ G = Q$.
The fixed points of $G(z)$ are $u^*$ and $v^*$.
\end{lemma}

\begin{proof}
The proof is immediate from the observations
\begin{align*}
aG^2(z)+2bG(z)+3c &= - \frac{3\Delta}{(2(a^2-3b)z+(ab-9c))^2} (az^2+2bz+3c), \\
3G^2(z)+2aG(z)+b &= - \frac{3\Delta}{(2(a^2-3b)z+(ab-9c))^2} (3z^2+2az+b),
\end{align*}
the definition of $Q(z)$, and (\ref{maybe-last-2}). 
\end{proof}

For any $z \in \CC^*$, let $f_{1}(z), f_{2}(z)$ be the pre-images of $z$ under $Q$.
In other words, we have
$$
Q(f_{1}(z))  = Q(f_{2}(z)) = z \mbox{ for all $z \in \CC^*$}
$$
and in particular
$$
f_{1}(u^*) = f_{2}(u^*) = v^* \quad \mbox{and} \quad  f_{1}(v^*) = f_{2}(v^*) = u^*.
$$
Thus, every bi-solution of $P$ is of the form $(f_{i}(z), f_{i}(z), z)$ for some $i \in \{1,2\}$ and some $z \in \CC$.

\begin{corollary}
For every $z \in \CC^*$, we have
$G(f_1(z)) = f_2(z).$
\end{corollary}

\begin{proof}
By Lemma~\ref{third-charact}, for every $z \in \CC^*$, we have
$$
z = Q(f_1(z)) = Q (G(f_1(z)))
$$
showing that $G(f_1(z))$ is in the pre-image of $z$ under $Q$.  
\end{proof}
In passing, we note that the triples
$$
\big\{(z, z, Q(z)) : z \in \CC^* \setminus \{(-a \pm \sqrt{a^2-3b})/3\} \big\}
$$
and
$$
\big\{(G(z), G(z), Q(z)) : z \in \CC^* \setminus \{(-a \pm \sqrt{a^2-3b})/3, -(ab-9c)/(2(a^2-3b))\} \big\}
$$
are two rational parametrizations of all bi-solutions of $P$. (That follows by the definition of the map $Q(z)$ and Lemma~\ref{third-charact}.)  The reason why three points are excluded from the parametrization is that 
$$
Q((-a \pm \sqrt{a^2-3b})/3)  = G(-(ab-9c)/(2(a^2-3b))) = \infty.
$$

Being a symmetric M\"obius transformation, that is an involution, $G(z)$ is similar to $-z$, see \cite[page 66]{HS:1979}.
The transformation that exhibits the similarity is the one sending the fixed points of $G(z)$, $u^*, v^*$ to $\infty, 0$. In other words, by (\ref{Winv-1}) we have 
$$
G(W(z)) = W(-z).
$$
Thus, $G$, being an involution, leaves invariant the circles of the hyperbolic pencil, call it $\mathcal{H}$, whose point circles are its fixed points $u^*$ and $v^*$.  Denote by $\mathcal{E}$, the elliptic pencil orthogonal to $\mathcal{H}$. (It consists of all circles through $u^*$ and $v^*$.) 

Incidentally, since  the M\"obius transformation $W$ sends the pencil $\mathcal{H}$ into the pencil of all circles with centre at the origin, 
we note that the circle through the zeros $\alpha_1$, $\alpha_2$, and $\alpha_3$ is a member of $\mathcal{H}$, see Figure~\ref{GeomConstr}.

Since $G$ is an involution, $G(z)$ is on the circle through $z, u^*$, and $v^*$ with $z$ and $G(z)$ being on different arcs with endpoints $u^*$ and $v^*$. In other words, for any circles $C_h^3 \in \mathcal{H}$ and $C_e^3 \in \mathcal{E}$ with intersection points $\{f_{3,1}, f_{3,2}\}$, we have $G(f_{3,1})=f_{3,2}$ and by Lemma~\ref{third-charact} we have $Q(f_{3,1})=Q(f_{3,2})$. Thus, $f_{3,1}$ and $f_{3,2}$ are the fixed points of the M\"obius transformation $z \mapsto F(z, Q(f_{3,1}))$. Hence, 
if $(z_1,z_2, z_3)$ is a solution of $P$ with $z_3:= Q(f_{3,1})$ (that is $z_2 = F(z_1, Q(f_{3,1}))$), then $z_1, z_2, f_{3,1}, f_{3,2}$ are co-circular with $z_1$ and $z_2$ being on 
different arcs with end points $f_{3,1}$ and $f_{3,2}$. The situation is illustrated on Figure~\ref{GeomConstr}
for the polynomial
$$
p(z) = (z-1)(z-(1-i))(z-(-1+i)).
$$
The orientation of the axis is easy to deduce from the displayed roots of $p(z)$.

\begin{figure}[htp]
\includegraphics[scale=0.19]{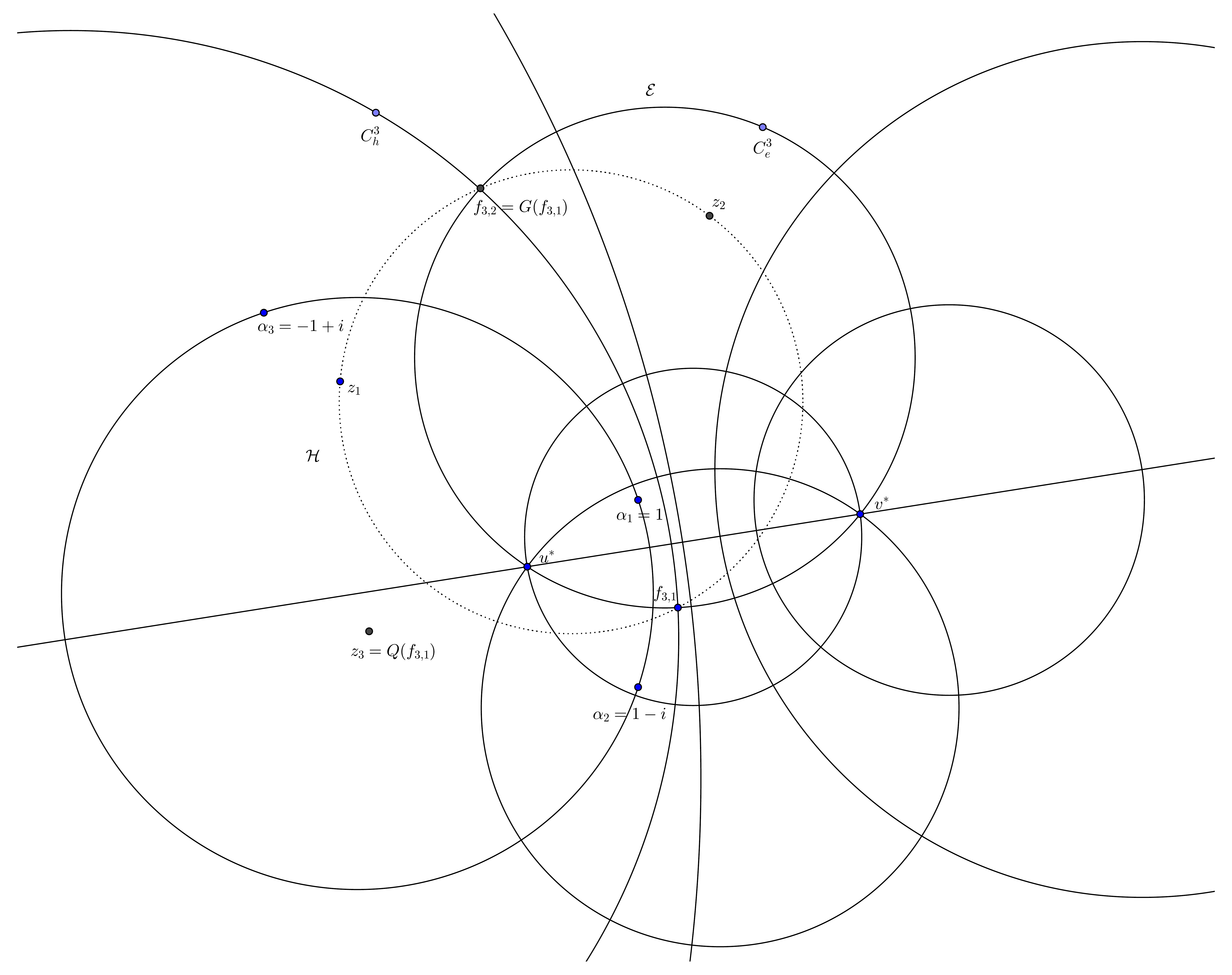}
\caption{Geometric construction of a solution of $P(z_1,z_2, z_3)=0$}
\label{GeomConstr}
\end{figure}

One can continue in a similar way to create the full picture described in Theorem~\ref{main-thm-2}.
The circles $C_h^k \in \mathcal{H}$ and $C_e^k \in \mathcal{E}$, depicted by a solid line on Figure~\ref{IllThm}, intersect at the points $\{f_{k,1}, f_{k,2}\}$, such that $G(f_{k,1})=f_{k,2}$ and  $Q(f_{k,1})=Q(f_{k,2}) = z_k$, for $k=1,2,3$. The triple $(z_1, z_2, z_3)$ is a solution of $P$. The circles $C_1, \ldots, C_7$ in Theorem~\ref{main-thm-2} are displayed with a dotted line, while the curve $W(N)$ is displayed with a dashed line. They all intersect at the point $W(u_0)$, where $u_0$ is defined by (\ref{defn-AB}) and $u_k := W^{-1}(z_k)$, for $k=1,2,3$.
 
\begin{figure}[htp]
\includegraphics[scale=0.25]{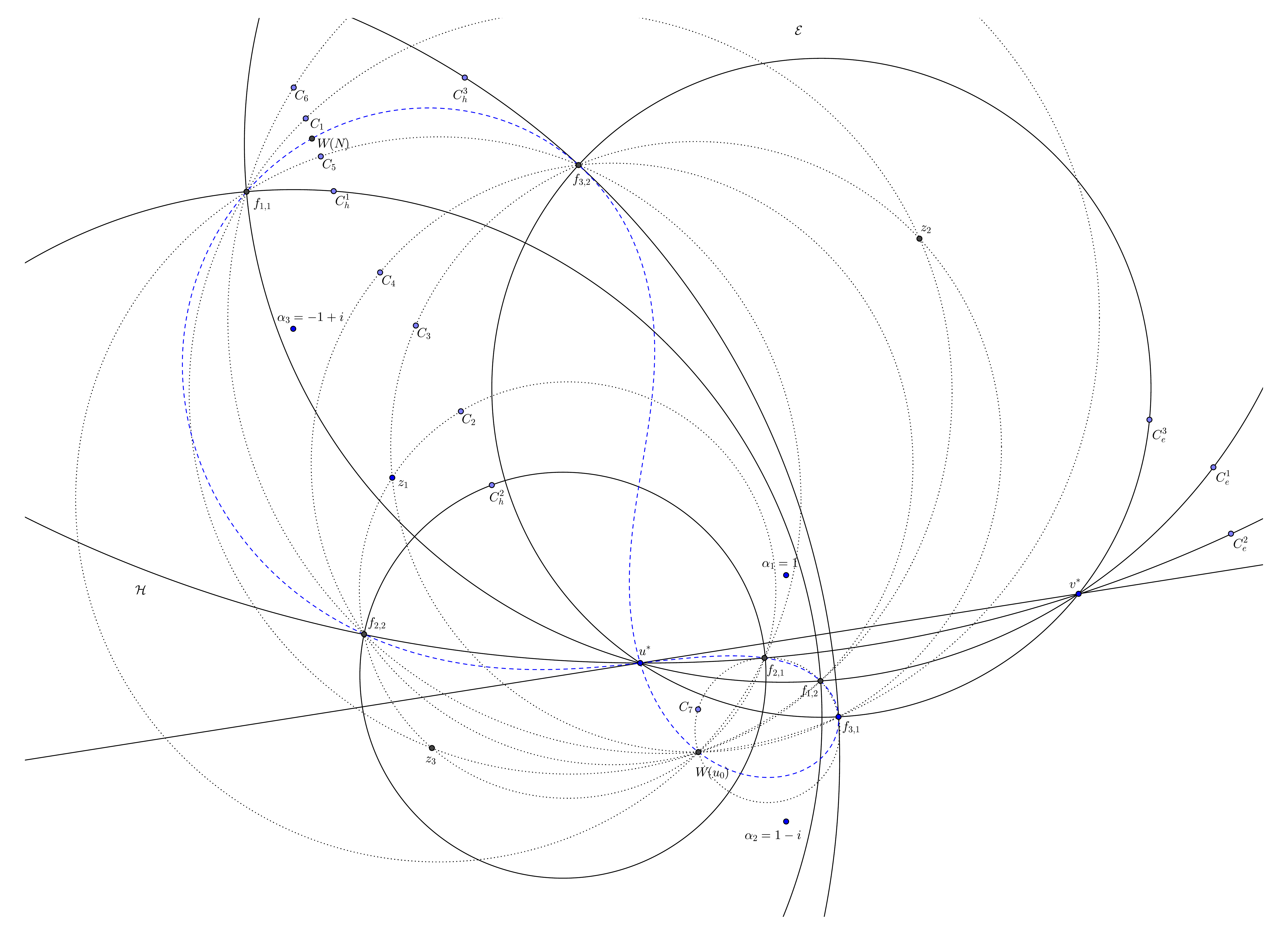}
\caption{Illustrating Theorem~\ref{main-thm-2}}
\label{IllThm}
\end{figure}

%**************************************************************************
\section{Proof of Theorem~\ref{main-thm-2}}
\label{sect-mainthm}
%**************************************************************************

The cross-ratio $Q$ of four points $u_1,u_2,u_3,u_4\in\CC$ is
\begin{equation}\label{eq1}
Q=\frac{u_1-u_3}{u_1-u_4} \frac{u_2-u_4}{u_2-u_3}.
\end{equation}
It is a well-known fact that the four points are co-circular if and only if $Q\in\mathbb{R}.$

Figure~\ref{7cir&1con} illustrates Theorem~\ref{scc}. On the figure, the common intersection point of the seven circles and the conic is denoted by $u_0$.

\begin{figure}[htp]
\includegraphics[scale=0.65]{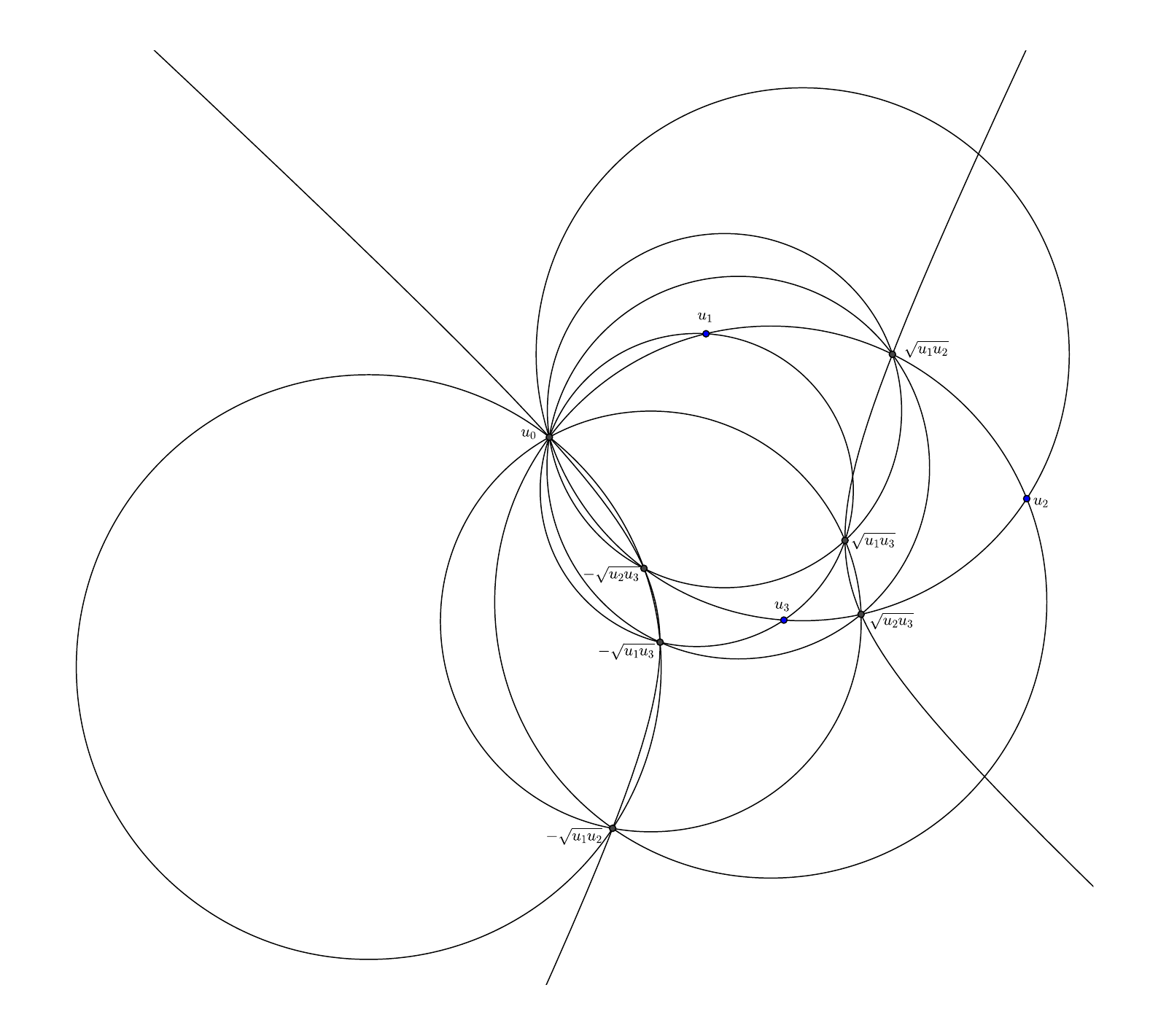}
\caption{Seven circles and a conic intersecting at one point}
\label{7cir&1con}
\end{figure}

\begin{theorem}
\label{scc}
Let $u_1,u_2,u_3$ be three non-zero, distinct points in the complex plane. Define seven circles and one conic  as follows.
\begin{itemize}
\item[$C_1:$] the circle determined by $u_2,u_3,\sqrt{u_2u_3};$
\item[$C_2:$] the circle determined by $u_1,u_3,\sqrt{u_1u_3};$
\item[$C_3:$] the circle determined by $u_1,u_2,\sqrt{u_1u_2};$
\item[$C_4:$] the circle determined by $-\sqrt{u_1u_2},-\sqrt{u_1u_3},-\sqrt{u_2u_3};$
\item[$C_5:$] the circle determined by $-\sqrt{u_1u_2},\sqrt{u_1u_3},\sqrt{u_2u_3};$
\item[$C_6:$] the circle determined by $\sqrt{u_1u_2},-\sqrt{u_1u_3},\sqrt{u_2u_3};$
\item[$C_7:$] the circle determined by $\sqrt{u_1u_2},\sqrt{u_1u_3},-\sqrt{u_2u_3};$
\item[$N_1:$] the conic determined by any five of $\pm\sqrt{u_1u_2},\pm\sqrt{u_1u_3},\pm\sqrt{u_2u_3}.$
\end{itemize}
Then all eight of these curves have a common intersection point. Moreover, the common intersection point is
\begin{align}
\label{defn-AB}
u_0:=\frac{\det A}{\det B},
\end{align}
where
\begin{align*}
A=\left[
\begin{array}{c c c}
\sqrt{\bar{u}_1 \bar{u}_2} & \sqrt{\bar{u}_1\bar{u}_3} & \sqrt{\bar{u}_2\bar{u}_3}\\
\sqrt{u_3} & \sqrt{u_2} & \sqrt{u_1}\\
\sqrt{\bar{u}_3} & \sqrt{\bar{u}_2} & \sqrt{\bar{u}_1}
\end{array}
\right],
~B= \left[
\begin{array}{c c c}
\bar{u}_1 & \bar{u}_2 & \bar{u}_3 \\[0.1cm]
\sqrt{\bar{u}_1 / u_1} & \sqrt{\bar{u}_2 / u_2} & \sqrt{\bar{u}_3 / u_3} \\[0.1cm]
1 & 1 & 1
\end{array}
\right].
\end{align*}
\end{theorem}

\begin{proof}
%Without loss of generality, suppose $z_1z_2z_3=-1.$ By M\"obius transformation,
%\begin{equation}\label{mt}
%w=\frac{-1}{zz_1},
%\end{equation}
%and we see that $(z_2,z_3)$ is a solution.
%Solving for the fixed points of equation (\ref{mt}), we get $z^2=\frac{-1}{z_1},$ which gives $z=\pm\sqrt{\frac{-1}{z_1}}=\pm\sqrt{z_2z_3}.$ Plugging $z_2,z_3,\pm\sqrt{z_2z_3}$ into equation (\ref{eq1}), we get
%\begin{align*}
%Q&=\frac{\sqrt{z_2z_3}-z_3}{\sqrt{z_2z_3}+\sqrt{z_2z_3}}\left(\frac{z_2-z_3}{z_2+\sqrt{z_2z_3}}\right)^{-1}\\
%&=\frac{\sqrt{z_2z_3}-z_3}{2\sqrt{z_2z_3}}\frac{z_2+\sqrt{z_2z_3}}{z_2-z_3}\\
%&=\frac{z_2\sqrt{z_2z_3}-z_2z_3+z_2z_3-z_3\sqrt{z_2z_3}}{2z_2\sqrt{z_2z_3}-2z_3\sqrt{z_2z_3}}\\
%&=\frac{(z_2-z_3)\sqrt{z_2z_3}}{2(z_2-z_3)\sqrt{z_2z_3}}\\
%&=\frac{1}{2}.
%\end{align*}
%Therefore, the four points are concircular. 
A point $u \in \CC$ is on the circle $C_1$, defined by the points, $u_2,u_3, \sqrt{u_2u_3}$, if and only if the cross ratio of 
$u_2,u_3, \sqrt{u_2u_3}$ and $u$ is a real number. Thus, substituting these four points into
equation (\ref{eq1}), we obtain the parametrized equation (with real parameter $Q_1$) of the circle $C_1.$ For ease of notation, let 
\begin{align}
\label{notation-wi}
w_i=\sqrt{u_i},~i\in\{1,2,3\}.
\end{align}
Then,
\begin{align*}
Q_1&=\frac{w_2w_3-w_3^2}{w_2w_3-u}\frac{w_2^2-u}{w_2^2-w_3^2}
%Q_1\frac{w_2^2-w_3^2}{w_2^2-z}&=\frac{w_2w_3-w_3^2}{w_2w_3-z},\\
%Q_1w_2w_3(w_2^2-w_3^2)-zQ_1(w_2^2-w_3^2)&=w_2^2(w_2w_3-w_3^2)-z(w_2w_3-w_3^2),\\
%z[w_2w_3-w_3^2-Q_1(w_2^2-w_3^2)]&=w_2^2(w_2w_3-w_3^2)-Q_1w_2w_3(w_2^2-w_3^2),\\
%z&=\frac{w_2^2(w_2w_3-w_3^2)-Q_1w_2w_3(w_2^2-w_3^2)}{w_2w_3-w_3^2-Q_1(w_2^2-w_3^2)}.
\end{align*}
and solving for $u$, we obtain the equation of $C_1$:
\begin{equation}\label{eq2}
u=\frac{w_2^2(w_2w_3-w_3^2)-Q_1w_2w_3(w_2^2-w_3^2)}{w_2w_3-w_3^2-Q_1(w_2^2-w_3^2)}, \quad Q_1 \in \RR.
\end{equation}
Similarly, we generate the equation of $C_2$:
\begin{equation}\label{eq3}
u=\frac{w_3^2(w_1w_3-w_1^2)-Q_2w_1w_3(w_3^2-w_1^2)}{w_1w_3-w_1^2-Q_2(w_3^2-w_1^2)}, \quad Q_2 \in \RR,
\end{equation}
and that of $C_3$:
\begin{align*}
%\label{eq4}
u=\frac{w_1^2(w_1w_2-w_2^2)-Q_3w_1w_2(w_1^2-w_2^2)}{w_1w_2-w_2^2-Q_3(w_1^2-w_2^2)}, \quad Q_3 \in \RR.
\end{align*}%\label{eq5}
To obtain the points of intersection of $C_1$ and $C_2$ we equate (\ref{eq2}) and (\ref{eq3}):
$$
\frac{w_2^2(w_2w_3-w_3^2)-Q_1w_2w_3(w_2^2-w_3^2)}{w_2w_3-w_3^2-Q_1(w_2^2-w_3^2)}=\frac{w_3^2(w_1w_3-w_1^2)-Q_2w_1w_3(w_3^2-w_1^2)}{w_1w_3-w_1^2-Q_2(w_3^2-w_2^2)},
$$
and solve for $Q_1$ and $Q_2$ as follows.  The last equation expands to
\begin{align*}
w_1w_2^3w_3^2 &-w_1^2w_2^3w_3-Q_2w_2^3w_3^3+Q_2w_1^2w_2^3w_3-w_1w_2^2w_3^3+w_1^2w_2^2w_3^2+Q_2w_2^2w_3^4 \\
&-Q_2w_1^2w_2^2w_3^2-Q_1w_1w_2^3w_3^2
+Q_1w_1^2w_2^3w_3+Q_1Q_2w_2^3w_3^3-Q_1Q_2w_1^2w_2^3w_3 \\
&+Q_1w_1w_2w_3^4-Q_1w_1^2w_2w_3^3-Q_1Q_2w_2w_3^5+Q_1Q_2w_1^2w_2w_3^3\\
=w_1&w_2w_3^4-w_1^2w_2w_3^3-Q_2w_1w_2w_3^4+Q_2w_1^3w_2w_3^2-w_1w_3^5+w_1^2w_3^4+Q_2w_1w_3^5 \\
&-Q_2w_1^3w_3^3 -Q_1w_1w_2^2w_3^3
+Q_1w_1^2w_2^2w_3^2+Q_1Q_2w_1w_2^2w_3^3-Q_1Q_2w_1^3w_2^2w_3 \\
&+Q_1w_1w_3^5-Q_1w_1^2w_3^4-Q_1Q_2w_1w_3^5+Q_1Q_2w_1^3w_3^3.
\end{align*}
This is an equation in two variables, the real parameters $Q_1$ and $Q_2,$ and by taking the complex conjugate of the above equation we obtain a second equation in the same two variables. Solving the resulting system of two equations for $Q_1$ and $Q_2,$ and then substituting back into, say (\ref{eq2}),  yields the intersection point $u_0$ of $C_1$ and $C_2$:
$$
\frac{w_1w_2w_3(-w_1\bar{w}_1\bar{w}_2^2+w_1\bar{w}_1\bar{w}_3^2-\bar{w}_1^2w_3\bar{w}_3+\bar{w}_2^2w_3\bar{w}_3+\bar{w}_1^2w_2\bar{w}_2-w_2\bar{w}_2\bar{w}_3^2)}{w_1w_2\bar{w}_2^2\bar{w}_3-w_1\bar{w}_2w_3\bar{w}_3^2+w_1\bar{w}_1^2\bar{w}_2w_3-w_1\bar{w}_1^2w_2\bar{w}_3+\bar{w}_1w_2w_3\bar{w}_3^2-\bar{w}_1w_2\bar{w}_2^2w_3}.
$$
(Here $\bar{w}_i$ is the complex conjugate of $w_i.$) Of course, the other intersection point between $C_1$ and $C_2$ is trivially $u_3$ by the definition of the circles.  Notice that the parenthetical portion of the numerator of $u_0$ is $\det A_\circ,$ and the denominator is $\det B_\circ,$ where 
$$
A_\circ=\left[\begin{array}{c c c}
\bar{w}_1\bar{w}_2&\bar{w}_1\bar{w}_3&\bar{w}_2\bar{w}_3\\
w_3&w_2&w_1\\
\bar{w}_3&\bar{w}_2&\bar{w}_1
\end{array}\right],~B_\circ= \left[\begin{array}{c c c}
w_1\bar{w}_1^2&w_2\bar{w}_2^2&w_3\bar{w}_3^2\\
\bar{w}_1&\bar{w}_2&\bar{w}_3\\
w_1&w_2&w_3
\end{array}\right].$$ 
In other words, we have
\begin{align}
\label{defn-z0}
u_0 = w_1w_2w_3\frac{\det A_\circ}{\det B_\circ}.
\end{align}
Equation~(\ref{defn-AB}) follows from here, after dividing $\det B_\circ$ by $w_1w_2w_3$ and reverting to the original variables $u_1,u_2,$ and $u_3$.
%Now putting $z_0$ in terms of the original points $z_1,z_2,z_3,$ we get
%\begin{align*}
%z_0&=\frac{\sqrt{z_1z_2z_3}(-\sqrt{z_1\bar{z}_1\bar{z}_2^2}+\sqrt{z_1\bar{z}_1\bar{z}_3^2}-\sqrt{\bar{z}_1^2z_3\bar{z}_3}+\sqrt{\bar{z}_2^2z_3\bar{z}_3}+\sqrt{\bar{z}_1^2z_2\bar{z}_2}-\sqrt{z_2\bar{z}_2\bar{z}_3^2})}{\sqrt{z_1z_2\bar{z}_2^2\bar{z}_3}-\sqrt{z_1\bar{z}_2z_3\bar{z}_3^2}+\sqrt{z_1\bar{z}_1^2\bar{z}_2z_3}-\sqrt{z_1\bar{z}_1^2z_2\bar{z}_3}+\sqrt{\bar{z}_1z_2z_3\bar{z}_3^2}-\sqrt{\bar{z}_1z_2\bar{z}_2^2z_3}}.\\
%&=\frac{z_1\sqrt{\bar{z}_1z_2z_3}(\bar{z}_3-\bar{z}_2)+z_2\sqrt{z_1\bar{z}_2z_3}(\bar{z}_1-\bar{z}_3)+z_3\sqrt{z_1z_2\bar{z}_3}(\bar{z}_2-\bar{z}_1)}{\sqrt{z_1\bar{z}_2\bar{z}_3|z_2|^2}-\sqrt{z_1\bar{z}_2\bar{z}_3|z_3|^2}+\sqrt{\bar{z}_1z_2\bar{z}_3|z_3|^2}-\sqrt{\bar{z}_1z_2\bar{z}_3|z_1|^2}+\sqrt{\bar{z}_1\bar{z}_2z_3|z_1|^2}-\sqrt{\bar{z}_1\bar{z}_2z_3|z_2|^2}}.
%\end{align*}

Upon inspection of the last expression for $u_0,$ one notes that it is invariant under every permutation of $u_1,u_2,u_3.$ This allows us to conclude that it does not matter which two of the three circles $C_1,$ $C_2,$ $C_3$ we choose to equate, the intersection point will be the same. That is, $C_1,$ $C_2$ and $C_3$ intersect at the point $u_0.$ 

To demonstrate that the other four circles intersect at $u_0$ as well, we need only show, for each circle, that the cross-ratio of the points defining the circle together with $u_0$ is a real number. We present the calculations for $C_4$ only. The other three cases are identical in procedure and the details are left out. Using notation (\ref{notation-wi}) and later on we let $w_i=u_i+iv_i,$ for $i=1,2,3$, the cross-ratio of the four points
$-\sqrt{u_1u_2},-\sqrt{u_1u_3},-\sqrt{u_2u_3},$ and $u_0$ is
\begin{align*}
Q&=\frac{u_0+w_1w_3}{u_0+w_1w_2}\left(\frac{-w_2w_3+w_1w_3}{-w_2w_3+w_1w_2}\right)^{-1}\\
&=\frac{(w_1w_2-w_2w_3)\big(w_1w_3+w_1w_2w_3 \det A_\circ / \det B_\circ\big)}{(w_1w_3-w_2w_3)\big(w_1w_2+w_1w_2w_3 \det A_\circ / \det B_\circ \big)}\\
%&=\frac{(w_1-w_3)(\bar{w}_1-\bar{w}_3)}{(w_1-w_2)(\bar{w}_1-\bar{w}_2)}\frac{w_1\bar{w}_1^2\bar{w}_2w_3-w_1\bar{w}_1^2w_2\bar{w}_3+w_1w_2\bar{w}_2^2\bar{w}_3-w_1\bar{w}_2w_3\bar{w}_3^2-\bar{w}_1w_2\bar{w}_2^2w_3+\bar{w}_1w_2w_3\bar{w}_3^2}{w_1\bar{w}_1^2\bar{w}_2w_3-w_1\bar{w}_1^2w_2\bar{w}_3+w_1w_2\bar{w}_2^2\bar{w}_3-w_1\bar{w}_2w_3\bar{w}_3^2-\bar{w}_1w_2\bar{w}_2^2w_3+\bar{w}_1w_2w_3\bar{w}_3^2}\\
&=\frac{(w_1-w_3)(\bar{w}_1-\bar{w}_3)}{(w_1-w_2)(\bar{w}_1-\bar{w}_2)}\\
&=\frac{w_1\bar{w}_1-\bar{w}_1w_3-w_1\bar{w}_3+w_3\bar{w}_3}{w_1\bar{w}_1-\bar{w}_1w_2-w_1\bar{w}_2+w_2\bar{w}_2}\\
&=\frac{|w_1|^2-(u_1-iv_1)(u_3+iv_3)-(u_1+iv_1)(u_3-iv_3)+|w_3|^2}{|w_1|^2-(u_1-iv_1)(u_2+iv_2)-(u_1+iv_1)(u_2-iv_2)+|w_2|^2}\\
&=\frac{|w_1|^2-u_1u_3-v_1v_3-i(u_1v_3-u_3v_1)-u_1u_3-v_1v_3+i(u_1v_3-u_3v_1)+|w_3|^2}{|w_1|^2-u_1u_2-v_1v_2-i(u_1v_2-u_2v_1)-u_1u_2-v_1v_2+i(u_1v_2-u_2v_1)+|w_2|^2}\\
&=\frac{|w_1|^2-2(u_1u_3+v_1v_3)+|w_3|^2}{|w_1|^2-2(u_1u_2+v_1v_2)+|w_2|^2}.
\end{align*}
The last expression shows that the cross-ration $Q$ is a real number.
Hence, $C_4$ contains the point $u_0.$ Similarly, one can show that the point $u_0$ is on the circles $C_5,$ $C_6$ and $C_7$ as well.

Finally, to show that the conic $N_1$ contains the intersection point $u_0,$ we use the general formula for a conic section:
$$
\alpha u^2+\bar{\alpha}\bar{u}^2+\beta u+\bar{\beta}\bar{u}+fu\bar{u}+e=0,
$$
where $\alpha$ and $\beta$ are complex, and $f$ and $e$ are real numbers. 
Our task is to find the coefficients $\alpha, \beta$, $f,$ and $e$ so that the conic passes through the points 
$\pm\sqrt{u_1u_2},\pm\sqrt{u_1u_3},\pm\sqrt{u_2u_3}.$
First we divide through by one of the real coefficients that is not equal to zero, say $f,$ to eliminate it as a variable. That is, we use the above equation with $f=1.$ Recalling (\ref{notation-wi}), we use any five of the six points (namely $\pm\sqrt{u_1u_2},\pm\sqrt{u_1u_3},\sqrt{u_2u_3}$) to obtain five equations in the five variables $a,b,c,d,e,$ where $\alpha=a+ib$ and $\beta=c+id:$ 
\begin{align*}
(a+ib)w_1^2w_2^2+(a-ib)\bar{w}_1^2\bar{w}_2^2+(c+id)w_1w_2+(c-id)\bar{w}_1\bar{w}_2+w_1w_2\bar{w}_1\bar{w}_2+e=0,\\
(a+ib)w_1^2w_3^2+(a-ib)\bar{w}_1^2\bar{w}_3^2+(c+id)w_1w_3+(c-id)\bar{w}_1\bar{w}_3+w_1w_3\bar{w}_1\bar{w}_3+e=0,\\
(a+ib)w_2^2w_3^2+(a-ib)\bar{w}_2^2\bar{w}_3^2+(c+id)w_2w_3+(c-id)\bar{w}_2\bar{w}_3+w_2w_3\bar{w}_2\bar{w}_3+e=0,\\
(a+ib)w_1^2w_2^2+(a-ib)\bar{w}_1^2\bar{w}_2^2-(c+id)w_1w_2-(c-id)\bar{w}_1\bar{w}_2+w_1w_2\bar{w}_1\bar{w}_2+e=0,\\
(a+ib)w_1^2w_3^2+(a-ib)\bar{w}_1^2\bar{w}_3^2-(c+id)w_1w_3-(c-id)\bar{w}_1\bar{w}_3+w_1w_3\bar{w}_1\bar{w}_3+e=0.
\end{align*}
Solving this system is a cumbersome task, but the equations are linear in $a,b,c,d,e$ and with a little work we find that
\begin{align*}
a&=\frac{a_1+a_2}{-2P}, & 
b&=\frac{i(b_1+b_2)}{2P},\\
c&=0, & 
d&=0, & \mbox{ and}\\
e&=\frac{e_1e_2}{P},
\end{align*}
where
\begin{align*}
a_1&=w_1\bar{w}_1^3\bar{w}_2\bar{w}_3(\bar{w}_2w_3-w_2\bar{w}_3)+w_2\bar{w}_2^3\bar{w}_1\bar{w}_3(w_1\bar{w}_3-\bar{w}_1w_3) \\
&+w_3\bar{w}_3^3\bar{w}_1\bar{w}_2(\bar{w}_1w_2-w_1\bar{w}_2),\\
a_2&=w_1^3\bar{w}_1w_2w_3(\bar{w}_2w_3-w_2\bar{w}_3)+w_2^3\bar{w}_2w_1w_3(w_1\bar{w}_3-\bar{w}_1w_3) \\
&+w_3^3\bar{w}_3w_1w_2(\bar{w}_1w_2-w_1\bar{w}_2),\\
b_1&=w_1\bar{w}_1^3\bar{w}_2\bar{w}_3(\bar{w}_2w_3-w_2\bar{w}_3)+w_2\bar{w}_2^3\bar{w}_1\bar{w}_3(w_1\bar{w}_3-\bar{w}_1w_3) \\
&+w_3\bar{w}_3^3\bar{w}_1\bar{w}_2(\bar{w}_1w_2-w_1\bar{w}_2),\\
b_2&=w_1^3\bar{w}_1w_2w_3(w_2\bar{w}_3-\bar{w}_2w_3)+w_2^3\bar{w}_2w_1w_3(\bar{w}_1w_3-w_1\bar{w}_3) \\
&+w_3^3\bar{w}_3w_1w_2(w_1\bar{w}_2-\bar{w}_1w_2),\\
e_1&=w_1\bar{w}_1w_2\bar{w}_2w_3\bar{w}_3,\\
e_2&=w_1\bar{w}_1w_2^2\bar{w}_3^2-w_1^2w_2\bar{w}_2\bar{w}_3^2+\bar{w}_1^2w_2\bar{w}_2\bar{w}_3^2 -w_1\bar{w}_1\bar{w}_2^2w_3^2+w_1^2\bar{w}_2^2w_3\bar{w}_3 \\
&-\bar{w}_1^2w_2^2w_3\bar{w}_3,\\
P&=w_1^2\bar{w}_1^2\bar{w}_2^2w_3^2-\bar{w}_1^2w_2^2\bar{w}_2^2w_3^2-w_1^2\bar{w}_2^2w_3^2\bar{w}_3^2+w_1^2w_2^2\bar{w}_2^2\bar{w}_3^2+\bar{w}_1^2w_2^2w_3^2\bar{w}_3^2 \\
&-w_1^2\bar{w}_1^2w_2^2\bar{w}_3^2.
\end{align*}
The fact that $\beta = c+id = 0$ shows that the conic passes through the sixth point $-\sqrt{u_2u_3}$ as well, since it is symmetric with respect to the origin.
%(Notice that $b$ has a factor of $i,$ meaning that $\alpha$ is a real number in this case.) 
A little bit more algebra shows that $\alpha=- a_1/P$ and $\bar{\alpha}=(-a_2+b_2)/2P.$ Now we have the general equation with the coefficients in terms of the $w_i$ and their conjugates,
\begin{align}
\label{eqn-conic-N1}
-\frac{a_1}{P}u^2+\frac{-a_2+b_2}{2P}\bar{u}^2+u\bar{u}+\frac{e_1e_2}{P}=0.
\end{align}
%The sixth point is verified to lie on $N_1$ by plugging it into the above equation and noting that it holds true. This is not surprising, as the conic is symmetric about the origin, and we do not present the details here. 
It remains to verify that $u_0$ also satisfies this equation. It is somewhat straightforward to substitute (\ref{defn-z0}) and
%\begin{align*}
%z_0&=\frac{w_1w_2w_3(-w_1\bar{w}_1\bar{w}_2^2+w_1\bar{w}_1\bar{w}_3^2-\bar{w}_1^2w_3\bar{w}_3+\bar{w}_2^2w_3\bar{w}_3+\bar{w}_1^2w_2\bar{w}_2-w_2\bar{w}_2\bar{w}_3^2)}{w_1w_2\bar{w}_2^2\bar{w}_3-w_1\bar{w}_2w_3\bar{w}_3^2+w_1\bar{w}_1^2\bar{w}_2w_3-w_1\bar{w}_1^2w_2\bar{w}_3+\bar{w}_1w_2w_3\bar{w}_3^2-\bar{w}_1w_2\bar{w}_2^2w_3}\\
%&=w_1w_2w_3\frac{\det A}{\det B}.
%\end{align*}
%Hence,
\begin{align*}
\bar{u}_0&=\bar{w}_1\bar{w}_2\bar{w}_3\frac{\det \bar{A}_\circ}{\det \bar{B}_\circ}.
\end{align*}
into (\ref{eqn-conic-N1}) and employ a program such as Maple to show that the left-hand side of this expression does indeed reduce to zero. This proves that $u_0$ lies on $N_1.$
%Substituting these values into the general equation, we get
%$$-\frac{a_1}{P}\left(w_1w_2w_3\frac{\det A}{\det B}\right)^2+\frac{-a_2+b_2}{2P}\left(\bar{w}_1\bar{w}_2\bar{w}_3\frac{\overline{\det A}}{\overline{\det B}}\right)^2+w_1w_2w_3\frac{\det A}{\det B}\bar{w}_1\bar{w}_2\bar{w}_3\frac{\overline{\det A}}{\overline{\det B}}+\frac{e_1e_2}{P}=0.$$
%Employing a program such as Maple shows that the left-hand side of this expression does indeed reduce to zero, proving that $z_0$ lies on $N_1.$
\end{proof}

If two of the points $u_1,u_2,u_3$ in the statement of Theorem~\ref{scc} are on the same circle with centre
at the origin (i.e. have equal modulus), then $u_0$ is also on that circle. Finally, if all three points $u_1,u_2,u_3$ are on the same circle with centre at the origin, then all seven circles and the conic coincide with that circle as well. Some related problems are considered in Examples 34 and 35 in \cite{Modenov:1981}.

A simple substitution 
\begin{align*}
w_1:= \sqrt{u_2u_3}, \,\,\, w_2:= \sqrt{u_1u_3}, \mbox{ and } w_3:=\sqrt{u_1u_2}.
\end{align*}
allows us to state Theorem~\ref{scc} in an alternative form.

\begin{corollary}
\label{scc-w}
Let $w_1, w_2, w_3$ be three non-zero, distinct points in the complex plane. Define seven circles and one conic  as follows.
\begin{itemize}
\item[$C_1:$] the circle determined by $w_1, w_1w_3/w_2,  w_1w_2/w_3;$
\item[$C_2:$] the circle determined by $w_2w_3/w_1, w_2, w_1w_2/w_3;$
\item[$C_3:$] the circle determined by $w_2w_3/w_1, w_1w_3/w_2, w_3;$
\item[$C_4:$] the circle determined by $-w_1, -w_2,-w_3;$
\item[$C_5:$] the circle determined by $w_1, w_2, -w_3;$
\item[$C_6:$] the circle determined by $w_1, -w_2, w_3;$
\item[$C_7:$] the circle determined by $-w_1, w_2, w_3;$
\item[$N_1:$] the conic determined by any five of $\pm w_1, \pm w_2, \pm w_3.$
\end{itemize}
Then all eight of these curves have a common intersection point. The formula for the common intersection point
may be obtained by substituting 
$$
u_1= w_2w_3/w_1, \,\,\,
u_2= w_1w_3/w_2, \mbox{ and }  u_3= w_1w_2/w_3
$$
into (\ref{defn-AB}).
\end{corollary}

Though Corollary~\ref{scc-w} appears simpler, an attempt to prove it along the lines of Theorem~\ref{scc} does not lead to 
shorter calculations. 

Alternatively, one may prefer to attack Corollary~\ref{scc-w} with tools from elementary algebraic geometry, such as the {\it automatic geometric theorem proving} machinery from \cite[Chapter 6, \S 4]{CLO:1992}. See also \cite{CHOU:1988}. Again the manipulations are not significantly simpler, and we hope that the chosen approach will appeal to a wider audience. 

%**************************************************************************
\section{Appendix A: supporting lemmas}
\label{appA-tech}
%**************************************************************************

 \begin{lemma}
 \label{2014-04-14-b}
 If the zeros $\alpha_1,\alpha_2, \alpha_3$ are not vertices of an equilateral triangle, then
 the system 
 \begin{align}
 \label{system*}
 \{P_1(z_1,z_2)=0, P_2(z_1,z_2) = 0\}
 \end{align}
 has a unique solution $(u^*, v^*)$. Otherwise, it has no solutions.
 \end{lemma}

\begin{proof}
Suppose $\alpha_1,\alpha_2, \alpha_3$ are not vertices of an equilateral triangle.
Let $(z_1, z_2)$ be such that $P_1(z_1,z_2)=0, P_2(z_1,z_2) = 0$, then 
\begin{align*}
0&= P_1(z_1, z_2) u^* + P_2(z_1, z_2) = P(z_1,z_2,u^*) \\
&=\frac{1}{3} \frac{(\alpha_1e_2+\alpha_2e_1+\alpha_3e_3)^2}{(\alpha_1e_1+\alpha_2e_2+\alpha_3e_3)} (z_1-v^*)(z_2-v^*),
%&= \frac{1}{3} \frac{(\alpha_1e_2+\alpha_2e_1+\alpha_3e_3)^2}{(\alpha_1e_1+\alpha_2e_2+\alpha_3e_3)} (z_1-v^*)z_3 -
%\frac{1}{3} \frac{(\alpha_1e_2+\alpha_2e_1+\alpha_3e_3)^2}{(\alpha_1e_1+\alpha_2e_2+\alpha_3e_3)} (z_1-v^*)v^*.
\end{align*}
where we used (\ref{repres-Pu^*v^*}). Lemma~\ref{lem-1.1} now implies that $z_1$ or $z_2$ is equal to $v^*$. Suppose $z_2=v^*$.
Then, for any $z_3 \in \CC \setminus \{u^*\}$, we have
\begin{align*}
0&= P_1(z_1, v^*) z_3 + P_2(z_1, v^*) = P(z_1,v^*,z_3) \\
&=\frac{1}{3} \frac{(\alpha_1e_1+\alpha_2e_2+\alpha_3e_3)^2}{(\alpha_1e_2+\alpha_2e_1+\alpha_3e_3)} (z_1-u^*)(z_3-u^*),
\end{align*}
where we used (\ref{repres-Pv^*u^*}). This shows that $z_1 = u^*$.

Suppose now $\alpha_1,\alpha_2, \alpha_3$ are vertices of an equilateral triangle.
By (\ref{ifequitriag}), we consider two cases.

a) If $ \alpha_1e_1+\alpha_2e_2+\alpha_3e_3=0$, then solving this equality for $\alpha_3$, one can verify the following statements:
$$
0 \not= -\frac{i}{3\sqrt{3}}(\alpha_1-\alpha_2)^3 =  P(z_1,z_2,v^*) = P_1(z_1,z_2)v^* + P_2(z_1,z_2).
$$
This shows that system (\ref{system*}) has no solution.

b) In the case $\alpha_1e_2+\alpha_2e_1+\alpha_3e_3 = 0$, the proof is analogous to part a), one only needs to consider $P(z_1,z_2,u^*)$.
\end{proof}
 
 \begin{corollary}
\label{lem-2014-02-25-a}
% maybe-last
 The system $\{P_1(z,z)=0, P_2(z,z) = 0\}$ has no solutions in $\CC$.
 \end{corollary}

\begin{lemma}
\label{lem-20140227}
If $\alpha_1$, $\alpha_2$, and $\alpha_3$ are vertices of an equilateral triangle, then
the only solutions of the system 
\begin{align}
\label{sys-zzw}
\{P(z,z,w)=0, P(z,w,w)=0\}
\end{align}
are  $(\alpha_1, \alpha_1), (\alpha_2, \alpha_2), (\alpha_3, \alpha_3)$. 
Otherwise, (\ref{sys-zzw}) has one more solution $(u^*, v^*).$
\end{lemma}

\begin{proof}
It is clear that $(\alpha_1, \alpha_1), (\alpha_2, \alpha_2), (\alpha_3, \alpha_3)$ are solutions and that these are the only solutions $(z,w)$ with equal components. So, suppose that $z \not= w$. Since
\begin{align*}
P(z, z, w)-P(z, w, w) = (z-w)(zw+(1/3)a(z+w)+(1/3)b),
\end{align*}
any solution $(z,w)$, with distinct components, satisfies
$$
zw+(1/3)a(z+w)+(1/3)b = 0.
$$
Multiplying this equation by $z$ and subtracting it from $P(z,z,w)=0$, we obtain
$$
(1/3)azw+(1/3)b(z+w)+c = 0.
$$
The last two equations form a system of linear equations for the quantities $zw$ and $z+w$. 
If $\alpha_1$, $\alpha_2$, and $\alpha_3$ are vertices of equilateral triangle, that is, $a^2 = 3b$, then the system has no solutions, as can be seen using (\ref{maybe-last}). Otherwise, the 
solution is
$$
zw= \frac{b^2-3ac}{a^2-3b} \mbox{ and } z+w = \frac{9c-ab}{a^2-3b}.
$$
Hence, $z$ and $w$ are the roots of the quadratic equation 
$$
(a^2-3b)t^2 + (ab-9c)t + (b^2-3ac) = 0.
$$
Hence, 
$$
z = \frac{(9c-ab)+i\sqrt{3} \sqrt{\Delta}}{2(a^2-3b)} \mbox{ and }
w = \frac{(9c-ab)-i\sqrt{3} \sqrt{\Delta}}{2(a^2-3b)}.
$$
This shows that the system $\{P(z,z,w)=0, P(z,w,w)=0\}$ has a unique solution with $z \not= w$.  Representations (\ref{repres-Pu^*v^*}) and (\ref{repres-Pv^*u^*}) show that it must be $(u^*, v^*)$.
\end{proof}

\begin{lemma}
The values $P_1(u^*, u^*), P_2(u^*, u^*), P_1(v^*, v^*),$ and $P_2(v^*, v^*)$ are all non-zero.
\end{lemma}

\begin{proof}
Expressing (\ref{repres-Pu^*v^*}) in the form
$$
P(z_1,z_2, u^*) = P_1(z_1, u^*) z_2 + P_2(z_1, u^*),
$$ 
shows that
\begin{align*}
P_1(u^*, u^*) &=  \frac{1}{3} \frac{(\alpha_1e_2+\alpha_2e_1+\alpha_3e_3)^2}{(\alpha_1e_1+\alpha_2e_2+\alpha_3e_3)} (u^*-v^*)
\quad \mbox{and }\\
P_2(u^*, u^*) &= - \frac{1}{3} \frac{(\alpha_1e_2+\alpha_2e_1+\alpha_3e_3)^2}{(\alpha_1e_1+\alpha_2e_2+\alpha_3e_3)} (u^*-v^*) v^*
\end{align*}
are non-zero. Similarly for $P_1(v^*, v^*)$ and $P_2(v^*, v^*)$.
\end{proof}

%************************************************************************************************************
\section{Appendix B: proof of Theorem~\ref{thm-reduction}}
\label{sect-proofThm}
%************************************************************************************************************

Let $W(z)=N(z)/D(z)$, where
\begin{align*}
N(z) &:= z(\alpha_1\alpha_2 e_3 + \alpha_1\alpha_3 e_2 + \alpha_2\alpha_3 e_1)+(\alpha_1\alpha_2 e_3 + \alpha_1\alpha_3 e_1 + \alpha_2\alpha_3 e_2), \\
D(z) &:= -z(\alpha_1 e_1 + \alpha_2 e_2 + \alpha_3 e_3) - (\alpha_1 e_2 + \alpha_2 e_1 + \alpha_3 e_3).
\end{align*}
Using the fact that $e_1+e_2+e_3 = 0$, we have
\begin{align}
N(z) - \alpha_1 D(z) &= z(\alpha_1\alpha_2 e_3 + \alpha_1\alpha_3 e_2 + \alpha_2\alpha_3 e_1 + \alpha_1\alpha_1 e_1 + \alpha_1\alpha_2 e_2 + \alpha_1\alpha_3 e_3 ) \nonumber \\
&\hspace{0.5cm}+ (\alpha_1\alpha_2 e_3 + \alpha_1\alpha_3 e_1 + \alpha_2\alpha_3 e_2 + \alpha_1 \alpha_1 e_2 + \alpha_1\alpha_2 e_1 + \alpha_1 \alpha_3 e_3) \nonumber \\
&= z(-\alpha_1\alpha_2 e_1 - \alpha_1\alpha_3 e_1 + \alpha_2\alpha_3 e_1 + \alpha_1\alpha_1 e_1 ) \nonumber \\
&\hspace{0.5cm} + (-\alpha_1\alpha_2 e_2 - \alpha_1\alpha_3 e_2 + \alpha_2\alpha_3 e_2 + \alpha_1 \alpha_1 e_2 ) \nonumber \\
&= (z e_1 + e_2)(-\alpha_1\alpha_2 - \alpha_1\alpha_3 + \alpha_2\alpha_3 + \alpha_1\alpha_1) \nonumber \\
&= (z e_1 + e_2)(\alpha_1-\alpha_2)(\alpha_1-\alpha_3). 
\label{ident-1-1-05-2012}
\end{align}
Similarly, we have
\begin{align}
\label{ident-2-1-05-2012}
N(z) - \alpha_2 D(z) &= (z e_2 + e_1)(\alpha_2 -\alpha_3)(\alpha_2-\alpha_1), \\
N(z) - \alpha_3 D(z) &= (z e_3 + e_3)(\alpha_3 -\alpha_1)(\alpha_3-\alpha_2).
\label{ident-3-1-05-2012}
\end{align}

Next, separate the numerator and the denominator of the rational function $F(z_1,z_2)$. That is, let 
$F(z_1,z_2) = F_1(z_1,z_2)/ F_2(z_1,z_2)$, where
\begin{align*}
F_1(z_1,z_2) &:=  z_1z_2 (\alpha_1+\alpha_2+\alpha_3) - (z_1+z_2)(\alpha_1\alpha_2+\alpha_1\alpha_3+\alpha_2\alpha_3) + 3\alpha_1\alpha_2\alpha_3, \\
F_2(z_1,z_2) &:= 3z_1z_2- (z_1+z_2)(\alpha_1+\alpha_2+\alpha_3)+(\alpha_1\alpha_2+\alpha_1\alpha_3+\alpha_2\alpha_3).
\end{align*} 
We need to show that
\begin{align}
\label{need-to-show}
\frac{F_1(N(z_1)/D(z_1), N(z_2)/D(z_2))}{F_2(N(z_1)/D(z_1),N(z_2)/D(z_2))} = \frac{N(-1/z_1z_2)}{D(-1/z_1z_2)}.
\end{align} 
Multiply the numerator and denominator on the left-hand side of (\ref{need-to-show}) by $D(z_1) D(z_2)$ and consider each one separately. Below we utilize identities (\ref{ident-1-1-05-2012}), (\ref{ident-2-1-05-2012}), and (\ref{ident-3-1-05-2012}) which hold for every $z \in \CC$. 
\begin{align}
F_1(&N(z_1)/D(z_1), N(z_2)/D(z_2)) D(z_1) D(z_2)  = N(z_1)N(z_2) (\alpha_1+\alpha_2+\alpha_3)  \nonumber \\
&\hspace{0.5cm} - (N(z_1)D(z_2)+N(z_2)D(z_1))(\alpha_1\alpha_2+\alpha_1\alpha_3+\alpha_2\alpha_3) + 3\alpha_1\alpha_2\alpha_3 D(z_1)D(z_2) \nonumber \\
&= \alpha_3 (N(z_2) - \alpha_1D(z_2)) (N(z_1) - \alpha_2 D(z_1)) \nonumber  \\
&\hspace{0.5cm} + \alpha_1 (N(z_2) - \alpha_2D(z_2)) (N(z_1) - \alpha_3 D(z_1)) \nonumber \\
&\hspace{1cm} + \alpha_2 (N(z_2) - \alpha_3 D(z_2)) (N(z_1) - \alpha_1 D(z_1)) \nonumber \\
&= \alpha_3 (z_2 e_1 + e_2)(\alpha_1-\alpha_2)(\alpha_1-\alpha_3) (z_1 e_2 + e_1)(\alpha_2 -\alpha_3)(\alpha_2-\alpha_1) \nonumber \\
&\hspace{0.5cm} + \alpha_1 (z_2 e_2 + e_1)(\alpha_2 -\alpha_3)(\alpha_2-\alpha_1)(z_1 e_3 + e_3)(\alpha_3 -\alpha_1)(\alpha_3-\alpha_2) \nonumber \\
&\hspace{1cm} + \alpha_2 (z_2 e_3 + e_3)(\alpha_3 -\alpha_1)(\alpha_3-\alpha_2)(z_1 e_1 + e_2)(\alpha_1-\alpha_2)(\alpha_1-\alpha_3) \nonumber  \\
&= (\alpha_1-\alpha_2)(\alpha_1-\alpha_3) (\alpha_2 -\alpha_3) \big(\alpha_3 (z_2 e_1 + e_2) (z_1 e_2 + e_1) (\alpha_2-\alpha_1)  \nonumber \\
&\hspace{0.5cm} + \alpha_1 (z_2 e_2 + e_1)(z_1 e_3 + e_3)(\alpha_3-\alpha_2) +  \alpha_2 (z_2 e_3 + e_3) (z_1 e_1 + e_2)(\alpha_1 -\alpha_3) \big) \nonumber \\
&\label{expr-boza}
\end{align}
Next, using $e_1-e_2 =  i\sqrt{3} e_3$, $e_1-e_3 = - i \sqrt{3} e_2$, and $e_2-e_3 = i \sqrt{3} e_1$ we have the identities
\begin{align*}
(z_2e_3+e_3)(z_1e_1+e_2)-(z_2e_2+e_1)(z_1e_3+e_3) &= z_1z_2(e_1e_3 - e_2e_3) + (e_2e_3 - e_1e_3) \\
&= -i\sqrt{3} (z_1z_2e_3 -e_3)  , \\
(z_2e_2+e_1)(z_1e_3+e_3)-(z_2e_1+e_2)(z_1e_2+e_1) &= z_1z_2(e_2e_3 - e_2e_1) +(e_1e_3-e_1e_2) \\
&=-i\sqrt{3} (z_1z_2e_1 - e_2), \\
(z_2e_1+e_2)(z_1e_2+e_1)-(z_2e_3+e_3)(z_1e_1+e_2) &= z_1z_2(e_1e_2-e_1e_3) + (e_1e_2-e_2e_3)\\
&= -i\sqrt{3} (z_1z_2e_2 - e_1).
\end{align*}
Substituting into expression (\ref{expr-boza}) we continue
\begin{align}
F_1(N(z_1)/&D(z_1), N(z_2)/D(z_2)) D(z_1) D(z_2) = -i\sqrt{3} (\alpha_1-\alpha_2)(\alpha_1-\alpha_3) (\alpha_2 -\alpha_3)  \times \nonumber \\
&\hspace{0.5cm} \times
\big( \alpha_1\alpha_2 (z_1z_2e_3 -e_3) + \alpha_1\alpha_3 (z_1z_2e_1 -e_2) + \alpha_2\alpha_3 (z_1z_2e_2 -e_1)\big) 
\nonumber \\
&= -i \sqrt{3} (\alpha_1-\alpha_2)(\alpha_1-\alpha_3) (\alpha_2 -\alpha_3) (z_1z_2) N(-1/z_1z_2).
\label{final-num}
\end{align}
For the denominator we have
\begin{align}
F_2(&N(z_1)/D(z_1), N(z_2)/D(z_2)) D(z_1) D(z_2)  = 3N(z_1)N(z_2)  \nonumber \\
&\hspace{0.5cm} - (N(z_1)D(z_2)+N(z_2)D(z_1))(\alpha_1+\alpha_2+\alpha_3)  \nonumber \\
&\hspace{1cm} +D(z_1)D(z_2) (\alpha_1\alpha_2+\alpha_1\alpha_3+\alpha_2\alpha_3) \nonumber \\
&= (N(z_2) - \alpha_1D(z_2)) (N(z_1) - \alpha_3 D(z_1)) \nonumber \\
&\hspace{0.5cm} + (N(z_2) - \alpha_2D(z_2)) (N(z_1) - \alpha_1 D(z_1)) \nonumber \\
&\hspace{1cm} + (N(z_2) - \alpha_3 D(z_2)) (N(z_1) - \alpha_2 D(z_1)) \nonumber \\
&=  (z_2 e_1 + e_2)(\alpha_1-\alpha_2)(\alpha_1-\alpha_3) (z_1 e_3 + e_3)(\alpha_3 -\alpha_1)(\alpha_3-\alpha_2) \nonumber \\
&\hspace{0.5cm} +  (z_2 e_2 + e_1)(\alpha_2 -\alpha_3)(\alpha_2-\alpha_1)(z_1 e_1 + e_2)(\alpha_1 -\alpha_2)(\alpha_1-\alpha_3) \nonumber \\
&\hspace{1cm} +  (z_2 e_3 + e_3)(\alpha_3 -\alpha_1)(\alpha_3-\alpha_2)(z_1 e_2 + e_1)(\alpha_2-\alpha_3)(\alpha_2-\alpha_1) \nonumber \\
&= (\alpha_1-\alpha_2)(\alpha_1-\alpha_3) (\alpha_2 -\alpha_3) 
\big( (z_2 e_1 + e_2)(z_1 e_3 + e_3)(\alpha_1 -\alpha_3) \nonumber \\
&\hspace{0.5cm} +  (z_2 e_2 + e_1)(z_1 e_1 + e_2)(\alpha_2 -\alpha_1) 
+   (z_2 e_3 + e_3)(z_1 e_2 + e_1)(\alpha_3-\alpha_2)\big). \nonumber \\
&\label{expr-boza-1}
\end{align}
Again using $e_1-e_2 =  i\sqrt{3} e_3$, $e_1-e_3 = - i \sqrt{3} e_2$, and $e_2-e_3 = i \sqrt{3} e_1$ we have the identities
\begin{align*}
 (z_2 e_1 + e_2)(z_1 e_3 + e_3) -  (z_2 e_2 + e_1)(z_1 e_1 + e_2) &= z_1z_2(e_1e_3 - e_1e_2) + (e_2e_3 - e_1e_2) \\
&= i\sqrt{3} (z_1z_2e_2 -e_1) , \\
(z_2 e_2 + e_1)(z_1 e_1 + e_2)-(z_2 e_3 + e_3)(z_1 e_2 + e_1) &= z_1z_2(e_1e_2 - e_2e_3) +(e_1e_2-e_1e_3) \\
&=i\sqrt{3} (z_1z_2e_1 - e_2), \\
(z_2 e_3 + e_3)(z_1 e_2 + e_1)-(z_2 e_1 + e_2)(z_1 e_3 + e_3) &= z_1z_2(e_2e_3-e_1e_3) + (e_1e_3-e_2e_3)\\
&= i\sqrt{3} (z_1z_2e_3 - e_3).
\end{align*}
Substituting into  (\ref{expr-boza-1}) we continue
\begin{align}
F_2(N(z_1)/&D(z_1), N(z_2)/D(z_2)) D(z_1) D(z_2) =  i \sqrt{3} (\alpha_1-\alpha_2)(\alpha_1-\alpha_3) (\alpha_2 -\alpha_3) \times \nonumber \\
&\hspace{0.5cm}\times (\alpha_1 (z_1z_2e_2 -e_1) + \alpha_2(z_1z_2e_1 - e_2) + \alpha_3(z_1z_2e_3 - e_3)) 
\nonumber \\
&\hspace{0.5cm} =  - i \sqrt{3} (\alpha_1-\alpha_2)(\alpha_1-\alpha_3) (\alpha_2 -\alpha_3)(z_1z_2) D(-1/z_1z_2).
\label{final-den}
\end{align}
The proof concludes after dividing (\ref{final-num}) by (\ref{final-den}).

%************************************************************************************************************

\end{document}